\documentclass[reqno]{amsart}

\usepackage{hyperref}
\usepackage{bbm}
\usepackage{amsfonts,amsmath,amsxtra,amsthm,amssymb,latexsym,enumerate}
\usepackage{ifpdf}
\usepackage{tikz}

\newcommand*{\mailto}[1]{\href{mailto:#1}{\nolinkurl{#1}}}
\newcommand{\arxiv}[1]{\href{http://arxiv.org/abs/#1}{arXiv:#1}}

\newcommand{\msc}[1]{\href{http://www.ams.org/msc/msc2010.html?t=&s=#1}{#1}}

\makeatletter
\@namedef{subjclassname@2020}{%
  \textup{2020} Mathematics Subject Classification}

\newcommand{\ack}{\section*{Acknowledgments}}

\newtheorem{theorem}{Theorem}[section]
\newtheorem{corollary}[theorem]{Corollary}

\newtheorem{hypothesis}{Hypothesis}[section]
\theoremstyle{definition}

\numberwithin{equation}{section}

\DeclareMathOperator{\dom}{dom}
  
\newcommand\N{{\mathbb{N}}}
\newcommand\R{{\mathbb{R}}}
\newcommand\C{{\mathbb{C}}}
\newcommand\Z{{\mathbb{Z}}}

\newcommand\cI{{\mathcal{I}}}
\newcommand\cG{{\mathcal{G}}}
\newcommand\cE{{\mathcal{E}}}
\newcommand\cV{{\mathcal{V}}}
\newcommand\cR{{\mathcal{R}}}

\newcommand\bH{{\mathbf{H}}}
\newcommand\rh{{\mathbf{h}}}

\newcommand\E{{\rm{e}}}
\newcommand\vol{{\rm{vol}}}
\newcommand\I{{\rm{i}}}

\newcommand\f{{\bf{f}}}
\newcommand\g{{\bf{g}}}

\def\wt#1{{{\widetilde #1} }}

\begin{document}

\title[]{A note on Spectral Analysis of Quantum graphs}

\author[N. Nicolussi]{Noema Nicolussi}
\address{Faculty of Mathematics\\ University of Vienna\\ 
Oskar-Morgenstern-Platz 1\\ 1090 Vienna\\ Austria}
\email{noema.nicolussi@univie.ac.at}
\curraddr{Institut f\"ur Mathematik\\
Universität Potsdam\\
Karl-Liebknecht-Str.~24-25\\ 14476 Potsdam\\ Germany}

\thanks{{\it Research supported by the Austrian Science Fund (FWF) 
under Grant No.~J4497~(N.N.)}}


\begin{abstract}
We provide an introductory review of some topics in spectral theory of Laplacians on metric graphs. We focus on three different aspects: the trace formula, the self-adjointness problem and connections between Laplacians on metric graphs and discrete graph Laplacians.
\end{abstract}

\maketitle

\section{Introduction}

\emph{Quantum graphs} are Schrödinger operators on metric graphs (i.e. discrete graphs where edges are identified with intervals of certain lengths), acting on edgewise smooth functions satisfying certain coupling conditions at the vertices. The most studied quantum graph is the \emph{Kirchhoff Laplacian}, which corresponds to a Laplacian without potential and provides the analog of the Laplace--Beltrami operator on Riemannian manifolds in this setting.

Originally introduced by Pauling in the 1930s in order to model free electrons in organic molecules, quantum graphs are connected   to several diverse branches of mathematics and mathematical physics, placing them at the intersection of many subjects in mathematics and engineering. Their key features include their use as simplified models of complicated quantum systems, the appearance of metric graphs in tropical and algebraic geometry (where they can be seen as non-Archimedean analogs of Riemann surfaces), and applications in physics, mathematical biology and material sciences (often based on a one-dimensional graph approximation of a thin wire-like material). In the last decades, they have been studied extensively from different perspectives, and we only refer to a brief selection of recent monographs and collected works \cite{bcfk06, bk13, dv, ekkst08, exko, gnutzmannsmilansky, kn, post} for an overview and further references.

In this expository note, we give a brief overview of some aspects in the spectral theory of the Kirchhoff Laplacian. We discuss three different topics: the \emph{trace formula} (Section~\ref{sec:Traceformula}), the \emph{self-adjointness problem} (Section~\ref{sec:Selfadjointness}) and \emph{connections between quantum graphs and discrete graph Laplacians} (Section~\ref{sec:DiscretevsContinuous}). Necessary definitions and prerequisites are collected in Section~\ref{sec:Preliminaries}.

\section{Preliminaries} \label{sec:Preliminaries}

\subsection{Metric graphs} \label{ss:MetricGraphs}

Let $\cG_d = (\cV,\cE)$ be an undirected {\em graph}, that is, $\cV$ is a finite or countably infinite set of vertices and $\cE$ is a finite or countably infinite set of edges. We call two vertices $u$, $v\in \mathcal{V}$ {\em neighbors} and write $u\sim v$ if there is an edge $e_{u,v}\in \mathcal{E}$ connecting $u$ and $v$. We will always assume that $\cG_d$ is \emph{simple} (no loops or multiple edges) and \emph{connected}\footnote{These assumptions could be removed with a little extra effort, however, we impose them in order to streamline the exposition.}.

For every vertex $v\in \mathcal{V}$, we denote by $\cE_v$  the set of edges incident to $v$. The {\em degree} of a vertex $v\in\cV$ is given by
\begin{equation}  \label{eq:combdeg}
\deg(v):= \#\{e|\, e\in\cE_v\} .
\end{equation}

The following assumption is imposed throughout the paper.

\begin{hypothesis}\label{hyp:locfin}
$\cG_d$ is \emph{locally finite}, that is, $\deg(v) < \infty$ for every $v \in \cV$.
\end{hypothesis}

Assigning each edge $e\in\cE$ a finite length $\ell(e) \in (0,\infty)$ and considering the corresponding {\em edge length functio}n $\ell \colon \cE\to(0, \infty)$, we can naturally associate with $(\cG_d, \ell ) = (\cV,\cE,  \ell )$ a metric space $\cG$: we identify each edge $e \in \cE$ with an interval $\cI_e = [0, \ell(e)]$ and then obtain a topological space $\cG$ by "glueing together" the intervals according to the incidence relations in $\cG_d$. The topology on $\cG$ is metrizable by the {\em natural path metric} $\varrho_0$ -- the distance between two points $x,y \in\cG$ is defined as the arc length of the ``shortest continuous path" connecting them (for infinite graphs, such a path does not necessarily exist and one needs to take the infimum over all such paths).

A \emph{metric graph} is a metric space $\cG$ arising from the above construction for some collection $(\cG_d, \ell) =(\cV, \cE, \ell)$. Conversely, a collection $(\cG_d, \ell)$ whose metric realization coincides with $\cG$ is called a \emph{model} of $\cG$. Clearly, any metric graph has infinitely many models (e.g., obtained by subdividing edges using vertices of degree two). However, for the rest of this article, {\em we will always consider a metric graph $\cG$ together with a fixed model}. Abusing slightly the notation, we will usually not distinguish between the two objects and also write $\cG = (\cV, \cE, \ell)$ or $\cG= (\cG_d, \ell)$.

A metric graph $\cG = (\cV, \cE, \ell)$ is called {\em finite}, if $\cG_d = (\cV, \cE)$ has finitely many edges and vertices, and {\em infinite} otherwise. Note that $\cG = (\cV, \cE, \ell)$ is finite exactly when $\cG$ is compact as a topological space, and hence we can equivalently speak about {\em compact} and {\em non-compact} metric graphs.

\subsection{The Kirchhoff Laplacian}  \label{ss:KirchhoffLaplacian}
Let $\cG = (\cV, \cE, \ell)$ be a metric graph. The metric space $\cG$ carries a natural {\em Lebesgue measure}, obtained from the Lebesgue measures of its interval edges. The associated $L^2$-space $L^2(\cG)$ consists of all (equivalence classes of almost everywhere defined) functions $f \colon \cG \to \C$ such that
\begin{equation} \label{eq:L^2norm}
\| f \|^2_{L^2(\cG)} := \int_\cG |f(x)|^2 \, dx < \infty.
\end{equation}
Equipped with the norm \eqref{eq:L^2norm}, the $L^2$-space $L^2(\cG)$ is naturally a Hilbert space. For every edge $e \in \cE$, let $H^2(e) = H^2([0, \ell(e)])$ be the standard {\em second order Sobolev space} on the interval edge $e \cong [0, \ell(e)]$. By definition, $H^2(e)$ consists of all continuously differentiable functions $f_e \colon [0, \ell(e)] \to \C$  whose distributional second derivative is given by an $L^2$-function $f_e''$ in  $L^2(e) = L^2([0, \ell(e)])$.

Consider a function $f \colon\cG \to \C$ whose restriction $f_e := f|_e$ is twice continuously differentiable on each edge $e \in \cE$. Taking the {\em edgewise second derivative}
\begin{equation} \label{eq:Delta}
	\Delta f :=  (f''_e)_{e \in \cE},
\end{equation}
we obtain a function defined almost everywhere on the metric graph $\cG$ (namely except in the vertices). This definition still makes sense if the restrictions $f_e$ belong to the Sobolev space $H^2(e)$ for all $e \in \cE$, and the second derivatives in \eqref{eq:Delta} are taken in the distributional sense.

The domain of the {\em Kirchhoff Laplacian} is given by
\begin{align*}
\dom(\bH) := \{ f \in L^2(\cG) | \,  & \text{$f$ is continuous,  }\text{$f_e \in H^2(e)$ for all $e \in \cE$, } \\
&\text{$f$ satisfies \eqref{eq:kirchhoff} for all $v \in \cV$ } \text{and $\Delta f \in L^2(\cG)$} \},
\end{align*}
where the so-called {\em Kirchhoff conditions} at a vertex $v  \in \cV$ are given by
\begin{align}\label{eq:kirchhoff}
\begin{cases} f\ \text{is continuous at}\ v,\\[1mm] 
\sum\limits_{e \in\cE_v} f_e'(v) = 0, \end{cases} 
\end{align}
and $f_e'(v)$ denotes the derivative of $f$ in $v$ taken along the incident interval edge $e \in \cE_v$. Finally we define the {\em Kirchhoff Laplacian} $\bH$ by
\[
\bH f := - \Delta f, \qquad f \in \dom(\bH).
\]
Altogether, $\bH \colon \dom(\bH) \subseteq L^2(\cG) \to L^2(\cG)$ is an unbounded, densely defined and closed operator in the Hilbert space $L^2(\cG)$.

Note that, formally, we have introduced the Kirchhoff Laplacian $\bH$ by using a fixed model $(\cV, \cE, \ell)$ of the metric graph $\cG$. However, one easily verifies that the definition of $\bH$ is independent of the concrete choice of the model.

\subsection{Properties of the Kirchhoff Laplacian} In what follows, we are interested in the spectral properties of the operator $\bH$. One of the most basic questions, which underlies all subsequent spectral analysis, is the following: {\em Is $\bH$ self-adjoint}? That is, does $\bH = \bH^\ast$ hold for the adjoint operator $\bH^\ast$ of $\bH$? In the affirmative case, $\bH$ is a non-negative, self-adjoint operator in $L^2(\cG)$ and we need to understand the properties of its spectrum $\sigma(\bH) \subseteq [0, \infty)$. Here and below, we denote by
\[
\varrho(A) = \{ \lambda \in \C | (A - \lambda \operatorname{Id}) \text{ has a bounded inverse  defined everywhere on $\mathcal{H}$} \}
\]
the \emph{resolvent set} of a self-adjoint linear operator $A \colon \dom(A) \subseteq \mathcal{H} \to \mathcal{H}$ in a Hilbert space $\mathcal{H}$. Moreover, $\sigma(A) := \C \setminus \varrho(A)$ is the \emph{spectrum } of $A$, and we have $\sigma(A) \subseteq \R$.

\smallskip

Analogous to compact Riemannian manifolds, it is well-known that the Kirchhoff Laplacian $\bH$ on a finite (equivalently, compact) metric graph is self-adjoint and its spectrum $\sigma(\bH)$ is purely discrete (see e.g. \cite[Theorem 3.1.1]{bk13}). The spectrum consists of an infinite sequence $0 = \lambda_0 < \lambda_1 \le \lambda_2 \le \dots$ of eigenvalues of finite multiplicity tending to $+ \infty$ (by convention, eigenvalues are counted with their multiplicity in the sequence). The lowest eigenvalue is given by $\lambda_0 = 0$ and its eigenspace consists of the constant functions on $\cG$.  The latter holds since $\cG$ is connected by assumption.
Moreover, the eigenvalues satisfy {\em Weyl's law}
\begin{equation} \label{eq:Weyl}
\# \Big \{ k \in \N | \, \sqrt{\lambda_k} \le \lambda \Big \} = \frac{\vol(\cG)}{\pi}  \lambda + O(1), \qquad \lambda \to + \infty,
\end{equation}
where $\vol(\cG) = \sum_{e \in \cE}  \ell(e)$ is the total volume of $\cG$.

\smallskip

The properties of the Kirchhoff Laplacian $\bH$ on {infinite metric graphs} are more rich and less studied. This situation should be compared to the case of non-compact Riemannian manifolds. We stress that, in particular, self-adjointness may fail (see also Section~\ref{sec:Selfadjointness}) and clearly other types of spectra can occur.

\section{The trace formula for finite metric graphs} \label{sec:Traceformula}
In this section we review the {\em trace formula} for finite metric graphs. The trace formula was discovered independently by Roth \cite{roth} and Kottos--Smilansky \cite{kottossmilansky1, kottossmilansky2} (see also \cite{bk13, k08}). It connects the spectral data of the Kirchhoff Laplacian to geometric data stemming from the closed paths in the graph.

\smallskip

In order to state the result, we introduce the following notions. Let $\cG = (\cV, \cE, \ell)$ be a finite metric graph. By Euler's formula, the first Betti number $\mathfrak{g}$ of $\cG$ (which coincides with the cyclomatic number of $\cG_d = (\cV, \cE)$, that is, the dimension of the cycle space of $\cG_d$) is given by
\[
\mathfrak{g} = \# \cE - \# \cV +1.
\]
(Recall that, by assumption, $\cG$ is connected.) A {\em closed (combinatorial) path} in $\cG_d = (\cV,\cE)$ is a finite sequence $(v_i)_{i=0}^N$ of vertices such that each $v_i$ is connected to $v_{i+1}$ by an edge $e_{v_i, v_{i+1}}$ and, moreover, $v_0 = v_N$. Note that backtracking is allowed, meaning we may use an edge twice in a row, transversing it in opposite directions. A {\em periodic orbit} $p$ is a closed path up to forgetting which vertex is the starting point, that is, an equivalence class of closed paths obtained by cyclically permuting the $v_i$'s. The {\em set of periodic orbits} on $\cG_d$ is denoted by $\mathcal{P}$.

A periodic orbit $p \in \mathcal{P}$ is called {\em primitive}, if it can not be obtained by repeatedly transversing a shorter periodic orbit. Each periodic orbit $p \in \mathcal{P}$ is a repetition of a unique periodic orbit with minimal length, which is then primitive. The latter is called the {\em primitive part} of $p$ and denoted by $\operatorname{prim}(p)$.

To each periodic orbit $p \in \mathcal{P}$ (represented by a closed path $(v_i)_{i=0}^N$), we associate the following quantities:
\begin{itemize}
\item its {\em (arc) length} $\ell(p)$ in the metric graph $\cG$, given by $\ell(p) = \sum_{i=1}^N \ell(e_{v_{i-1}, v_{i}})$.
\item its {\em scattering coefficient} $s(p)$, given by the product of vertex scattering coefficients along the path,
\[
s(p) = \prod_{i=1}^N S_{v_i}(e_{v_{i-1}, v_{i}}, e_{v_{i}, v_{i+1}})
\]
(by convention, $v_{N+1} := v_1$). Here, for two edges $e,e'$ incident to a vertex $v \in \cV$, the vertex scattering coefficient is defined by
\[
S_v(e,e') =\frac{2}{\deg(v)} - \delta_{e,e'},
\]
where $\delta_{e,e'}$ is the Kronecker delta.

\end{itemize}

Let $\bH$ be the Kirchhoff Laplacian on $\cG$ and denote by $(\lambda_k)_{k=0}^\infty$ the increasing sequence of its eigenvalues (counted with multiplicity). Consider the following discrete measure on the real line (aka "counting measure of the wave spectrum")
\begin{equation} \label{eq:measure}
\mu = (\mathfrak{g}+1) \delta_0 + \sum_{k=1}^\infty \delta_{\sqrt{\lambda_k}} +  \delta_{- \sqrt{\lambda_k}},
\end{equation}
where $\delta_x$ is the Dirac measure centered at $x \in \R$. Taking into account Weyl's law~\eqref{eq:Weyl}, $\mu$ defines a tempered distribution. The trace formula expresses the Fourier transform of $\mu$ (taken in the distributional sense) in terms of the periodic orbits.

\begin{theorem}[see \cite{kottossmilansky2, k08, roth}]
Let $\cG = (\cV, \cE, \ell)$ be a finite metric graph. Then, in the sense of distributions,
\begin{equation} \label{eq:TraceFormula}
\widehat{ \mu} = 2 \vol(\cG)  \delta_0 + \sum_{p \in \mathcal{P}} s(p) \ell(\operatorname{prim}(p)) \Big ( \delta_{\ell(p)} + \delta_{- \ell(p)} \Big ),
\end{equation}
where $\vol(\cG) = \sum_{e \in \cE} \ell(e)$. That is, for all functions $f \colon \R \to \R$ belonging to the Schwartz space of rapidly decaying functions $\mathcal{S}(\R)$,
\begin{gather*}
(\mathfrak{g}+1) \hat f(0) + \sum_{k=1}^\infty \hat f({\sqrt{\lambda_k}}) + \hat f(- \sqrt{\lambda_k})  \\=  2\vol(\cG)  f (0)  + \sum_{p \in \mathcal{P}}    s(p) \ell(\operatorname{prim}(p))    \Big (  f(\ell(p)) +  f(- \ell(p)) \Big ),
\end{gather*}
where the Fourier transform is normalized as $\hat f (\xi) = \int_\R e^{- i \xi x} f(x) dx$, $\xi \in \R$.
\end{theorem}

The trace formula is a cornerstone for many developments around spectral theory of metric graphs, see e.g. \cite{bk13, gnutzmannsmilansky, gs,  kottossmilansky2} and the references therein. In the following, we briefly discuss three different directions from which the result can be viewed.

\subsection{Crystalline measures}  Applying the trace formula \eqref{eq:TraceFormula} to a circle $C$ of length one we recover a well-known result - the {\em Poisson summation formula}
\begin{equation} \label{eq:Poisson}
	\sum_{k \in \Z} \hat f( 2 \pi k) = \sum_{k \in \Z} f(k),
\end{equation}
which holds for all functions $f$ in the Schwartz space $\mathcal{S}(\R)$. Indeed, the circle $C$ can be viewed as a metric graph (e.g., an $n$-cycle with all edge lengths equal to $1/n$) and the spectrum of the Kirchhoff Laplacian is generated by an arithmetic progression $\sigma(\bH) = \{ (2\pi k)^2| \, k \in \Z \}$, where all non-zero eigenvalues have multiplicity two.

One may then pose the following question:
\begin{center} {\em What kind of "Poisson formulas" exist?}
\end{center} This leads to the following notion: a {\em crystalline measure} is a measure of the form $\mu = \sum_{x \in X} a_x \delta_x$ on $\R$ (with real or complex coefficients $a_x$) such that $\mu$ is a tempered distribution, its Fourier transform is again of the form $\hat \mu = \sum_{\lambda \in \Lambda} b_\lambda \delta_\lambda$ and the sets $X, \Lambda \subseteq \R$ are discrete \cite{meyer} (see also \cite[page 215]{dyson} for its connection to the Riemann hypothesis). By \eqref{eq:Poisson}, arithmetic progressions lead to crystalline measures, and conversely, it is known that every crystalline measure satisfying additional conditions stems from arithmetic progressions. On the other hand, Kurasov and Sarnak  recently obtained a detailed description of the arithmetic structure of metric graph spectra, which, naively speaking, says that generically $\sigma(\bH)$ is very far from an arithmetic progression \cite{ks, sarnaktalk}. In particular, for almost every metric graph, the measure $\mu$ in \eqref{eq:measure} provides a rather "unexpected" example of a crystalline measure (note that the supports of $\mu$ and $\hat \mu$ are discrete since $\sigma(\bH) \subseteq\R$ is discrete and there are only finitely many periodic orbits with length smaller than a fixed number). Whereas there are several other constructions of crystalline measures (see the references in \cite{ks, meyer21, ou20, radchenkoviazovska}), metric graphs provide the first non-trivial examples of so-called positive Fourier quasi-crystals (note that all coefficients of Dirac measures in \eqref{eq:measure} are positive) and have additional properties answering several open questions \cite{ks}. For further information and references, we refer to the recent works \cite{ks, meyer21, ou20}.

\subsection{Metric graphs and Riemann surfaces} Another viewpoint on analysis on metric graphs stems from the appearance of metric graphs (sometimes under the name tropical curve) in algebraic and tropical geometry, see e.g.\ \cite{amini, bj, baru10, dv}. For instance, they appear in context with \emph{degenerating families of Riemann surfaces}. In simplified terms, one considers a "nice" family $X_t$, $t >0$, of compact Riemann surfaces $X_t$ (these are analytifications of smooth complex algebraic curves), which in the limit $t \to 0$ converge to a nodal Riemann surface $X_0$. The latter corresponds to a singular curve with the mildest type of singularities. The so-called \emph{dual graph} $\cG_d = (\cV, \cE)$ of the singular Riemann surface $X_0$ captures its combinatorics: the vertices are the irreducible components of $X_0$ and the edges between two vertices correspond to intersection points of those two irreducible components. The graph $\cG_d$ can be equipped with suitable edge lengths describing the degeneration, leading to a metric graph $\cG$ (see Figure~\ref{fig:RiemannSurfacesDegeneration}).
\begin{figure} \label{fig:RiemannSurfacesDegeneration}
\begin{center}
\includegraphics[width= 0.6 \textwidth]{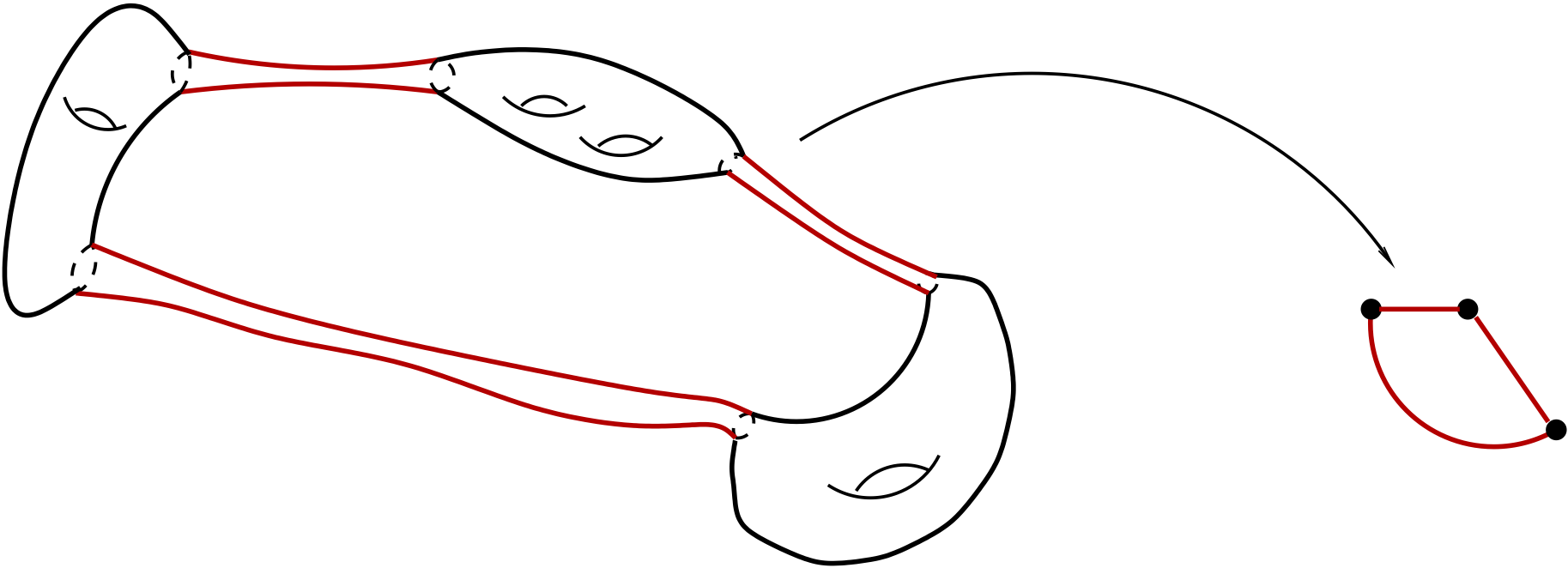}
\end{center}
\caption{A Riemann surface $X_t$, which is close to being singular, and a metric graph emerging from the degeneration in the limit.}
\end{figure}

Many classical algebraic-geometric and complex-analytic results on Riemann surfaces have a metric graph counterpart, see e.g.\ \cite{bano07, bano09, dv, mukzha}. For works studying the convergence of analytic objects on degenerating Riemann surfaces to their metric graph versions from an analytic point of view, we only refer to some very recent papers \cite{an20, an22, dejong, shivaprasad, wilms21} and the references therein.

Turning to spectral theory, from this perspective it seems natural to view the trace formula for compact metric graphs as an analog of {\em Selberg's trace formula for compact hyperbolic Riemann surfaces}, which relates the eigenvalues of the Laplace--Beltrami operator to closed oriented geodesics and their lengths \cite{iwaniec}. Note however that, when the Riemann surfaces degenerate, their smallest eigenvalues converge (after rescaling) to eigenvalues of a Laplacian on a discrete  (not metric!) graph \cite{colbois, colboiscdV}. Another interesting spectral parallel concerns an inequality by Yang--Yau \cite{yangyau}, which estimates the smallest non-zero eigenvalue of a compact Riemann surface by its gonality (the minimum degree of a holomorphic map to the Riemann sphere $\mathbb{P}^1$). Analogs for metric and discrete graphs were obtained in \cite{ak, ckk}, where gonality can be defined using harmonic maps from graphs to trees \cite{dv} (trees provide graph analogs of $\mathbb{P}^1$ since they are the graphs with first Betti number $\mathfrak{g} = 0$).

\subsection{Inverse spectral problems} The trace formula turns out to be useful in context with {\em inverse problems} in spectral geometry of metric graphs. Recall that the famous question {\em "Can One Hear the Shape of a Drum?"} (e.g.,  the title of a well-known article by M.\ Kac \cite{kac}) asks if a geometric object is uniquely determined by the spectrum of an associated (Laplacian) operator. A possible metric graph version of this question reads as follows: {\em "does the spectrum $\sigma(\bH)$ of the Kirchhoff Laplacian determine the underlying finite metric graph $\cG$?"} In case that the edge lengths $\ell(e)$, $e \in \cE$, are linearly independent over $\mathbb{Q}$, this indeed turns out to be the case (see \cite{gs, kn05}). The proofs in \cite{gs, kn05} are direct applications of the trace formula, and essentially recover $\cG$ from the right-hand side of \eqref{eq:TraceFormula}. However, analogous to Laplacians on Riemannian manifolds or discrete graphs, the answer turns out to be negative for general metric graphs. We refer to \cite[Section 7.1]{bk13} for further references on constructions of isospectral metric graphs and related problems.

\section{Self-adjointness} \label{sec:Selfadjointness}

The aim of this section is to discuss self-adjointness and Markovian uniqueness of the Kirchhoff Laplacian $\bH$ on a metric graph $\cG = (\cV, \cE, \ell)$.

\subsection{Self-adjointness and Markovian uniqueness} We begin by recalling basic facts about self-adjointness and Markovian uniqueness (see also \cite[Chapter 1]{ebe}).

\smallskip

 {\em Self-adjointness} is the first mathematical problem arising in any quantum mechanical model (see, e.g., \cite[Chap.~VIII.11]{RSI}). Namely, usually a formally symmetric expression for the Hamiltonian has some natural domain of definition in a given Hilbert space $\mathcal{H}$ and then one has to verify that it gives rise to a self-adjoint operator $A \colon \dom(T) \subseteq \mathcal{H} \to \mathcal{H}$, that is, the equality $A^\ast = A$ holds \footnote{We recall that, for a densely defined operator $A \colon \dom(A) \subseteq \mathcal{H} \to \mathcal{H}$ in a Hilbert space $\mathcal{H}$, its adjoint operator $A^\ast \colon \dom(A^\ast) \subseteq \mathcal{H} \to  \mathcal{H}$ is defined by introducing the domain of definition $\dom(A^\ast) = \{g \in \mathcal{H}| \, \exists h \in \mathcal{H} \text{ with }(Af, g)_{\mathcal{H}} = (f, h)_{\mathcal{H}} \forall f \in \dom(A) \}$ and then setting $A^\ast g := h$ for $g \in \dom(A^\ast)$.}. This property is closely connected to the Cauchy problem for the {\em Schr\"odinger equation}
\begin{align}\label{eq:I.schrod}
\I \partial_t u(t) & = A u(t), & u|_{t=0} &= u_0 \in \mathcal{H}.
\end{align}
It is exactly the self-adjointness of $A$ which ensures the existence and uniqueness of solutions to \eqref{eq:I.schrod}. Otherwise, there are infinitely many self-adjoint restrictions\footnote{Here we suppose that $A_0  := A^\ast$ is a symmetric operator with equal deficiency indices, which holds true in our context. Note also that our formulation is slightly non-standard, since we consider the maximal operator $A$. Equivalently and perhaps more common, one can start with a closed, symmetric operator $A_0 $ (the minimal operator) and ask for its self-adjoint extensions, which are exactly the self-adjoint restrictions of $A = A_0^\ast$. In our context the minimal operator $\bH_0 = \bH^\ast$ is obtained by restricting $\bH$ to compactly supported functions and taking the closure in $L^2(\cG)$. However, in order to keep the exposition short, we do not to introduce this additional operator.} $\wt A \colon \dom(\wt A) \subseteq \mathcal{H} \to \mathcal{H}$ (obtained by restricting $A$ to suitable smaller domains $\dom(\wt A) \subseteq \dom( A)$) and one has to choose the right one modeling the phenomenon in question, the so-called observable. In the context of differential operators on a geometric domain with boundary, this often corresponds to imposing {\em boundary conditions}. On the other hand, a self-adjoint operator has no non-trivial self-adjoint restrictions, and for this reason self-adjointness is sometimes also called {\em self-adjoint uniqueness}. For instance, the Laplacian on a bounded domain $\Omega \subseteq \R^N$ is not self-adjoint. Imposing Dirichlet or Neumann boundary conditions, respectively, leads to two different self-adjoint restrictions and hence Schrödinger equations. 
Naively speaking, self-adjointness of the Kirchhoff Laplacian $ \bH$ in the $L^2$-space $ L^2(\cG)$ means that, in the sense of quantum mechanics, there is a \emph{"unique meaningful Schrödinger equation on $\cG$"} (associated to $\bH$).

\smallskip

On the other hand, when studying a {\em diffusion process} on some metric measure space $X$  (e.g., Brownian motion on a Riemannian manifold), one is led to the Cauchy problem for the {\em heat equation} 
\begin{align}\label{eq:heat}
\partial_t u(t) & = A u(t), & u|_{t=0} &= u_0 \in L^2(X).
\end{align}
In order to ensure that solutions have properties reflecting heat diffusion, one considers \eqref{eq:heat} for so-called {\em Markovian operators}. More precisely, we require that  $A \colon \dom(A) \subseteq L^2(X) \to L^2(X)$ is self-adjoint, non-negative and the semigroup $(\E^{- t A})_{t>0}$ is positivity preserving and $L^\infty$ contractive (i.e., for a function $0 \le f \le 1$ in $L^2(X)$, we have $0 \le \E^{- t A}f \le 1$ for all $t >0$). By the Beurling--Deny criteria, $A$ is Markovian exactly when its quadratic form is a {\em Dirichlet form} (see e.g.\ \cite{dav89}).  

Analogous to Laplacians on Riemannian manifolds and discrete graph Laplacians, the Kirchhoff Laplacian $\bH$ on a metric graph $\cG$ has the following properties. If $\bH$ is self-adjoint, it is automatically Markovian. If $\bH$ is not self-adjoint, then it has both Markovian and non-Markovian self-adjoint restrictions. By {\em Markovian uniqueness} we mean that the Kirchhoff Laplacian $\bH$ has a unique Markovian restriction\footnote{Equivalently, the corresponding minimal Kirchhoff Laplacian $\bH_0 = \bH^\ast$ has a unique Markovian extension.}. Self-adjointness clearly implies Markovian uniqueness, but the opposite direction does not hold. Naively speaking, Markovian uniqueness means that, in the sense of diffusion processes, there is a \emph{"unique meaningful heat equation on $\cG$"} (associated to $\bH$).

\subsection{Self-adjointness criteria}\label{ss:SelfadjointnessCriteria} We begin by reviewing sufficient conditions for self-adjointness. Recall that in the natural path metric $\varrho_0$ on the metric graph $\cG$, the distance $\varrho_0(x,y)$ between two points $x,y \in \cG$ is the arc length of the shortest continuous path connecting them (see Section~\ref{ss:MetricGraphs}). The so-called {\em star metric} $\varrho_m$ is obtained by changing the arc length of an edge $e_{u,v}$ connecting $u,v \in \cV$ to 
\[
\ell_m(e_{u,v}) := m(u) + m(v) :=  \sum_{e \in \cE_u} \ell(e) + \sum_{e \in \cE_v} \ell(e),
\]
and again taking lengths of shortest connecting paths (now w.r.t.\ to $\ell_m$).

The following conditions imply that the Kirchhoff Laplacian $\bH$ is self-adjoint:
\begin{itemize}
\item [(i)] $\cG$ is finite
\item [(ii)] edge lengths are bounded from below, $\inf_{e \in \cE} \ell(e) > 0$ (\cite[Theorem 1.4.19]{bk13})
\item [(iii)] ($\cG, \varrho_0)$ is a complete metric space (\cite[Theorem 3.49]{hae} or  \cite[Corollary 4.9]{ekmn})
\item [(iv)] $(\cG, \varrho_m)$ is a complete metric space (\cite[Theorem 4.11]{ekmn})
\end{itemize}
It is easily shown that each condition  is weaker than the previous one. Condition (iii) is an analog of a result for manifolds, which ensures that the Laplace--Beltrami operator on complete Riemannian manifolds is self-adjoint \cite{roe}, whereas condition (iv) does not seem to have a manifold counterpart (for further discussion, see also \cite[Section 7]{kn}). On the other hand, simple examples show that none of the above conditions are necessary and obtaining a complete characterization of self-adjointness is probably quite difficult, see e.g. \cite[Remark 4.12 and Section~7]{kmn22}.

\subsection{Markovian uniqueness and graph ends}
In contrast to self-adjointness, Markovian uniqueness on metric graphs admits an explicit geometric characterization. In the remainder of the section, we describe the results of Kostenko, Mugnolo and the author from \cite{kmn22} (see also \cite{kn21} and \cite[Section 7.2]{kn}). In what follows, let $\cG = (\cV, \cE, \ell)$ be an infinite metric graph (as discussed above, for finite metric graphs, self-adjointness and hence Markovian uniqueness always holds).

\smallskip

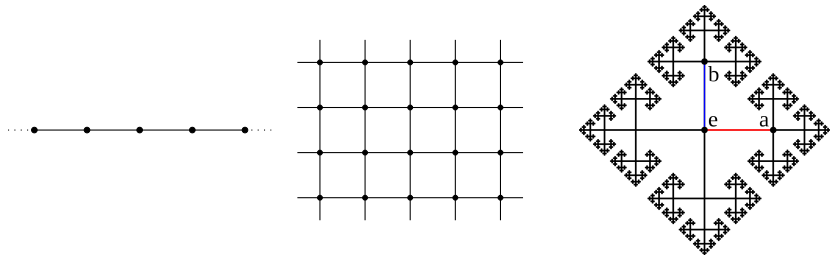
\begin{figure}[h!]  \label{fig:GraphEnds}
\begin{center}
\begin{minipage}{.3 \textwidth}
		\centering	
	\begin{tikzpicture} [scale = .35]
	\foreach \x in {0,...,4}
	{ \filldraw  (-4 + \x*2, 0) circle (3 pt);  }
	\foreach \x in {0,...,3}{
		\draw (-4 + 2* \x, 0) -- (-2 + 2* \x, 0);
	}
	\draw [dotted] (-5  , 0) -- (-4, 0);
	\draw [dotted] (4 , 0) -- (5,  0);
\end{tikzpicture}
\end{minipage} 
\begin{minipage}{.25 \textwidth}
	\centering
	\begin{tikzpicture} [scale = .3]
	\foreach \x in {0,...,4}
	\foreach \y in {0,...,3}
	{ \filldraw  (-4 + \x*2, 2*\y) circle (3 pt);  }
	\foreach \y in {0,...,3}{
		\draw (-4, 2* \y) -- (4, 2* \y) ;
		\draw   (-5, 2* \y) --(-4, 2* \y);
		\draw   (4, 2* \y) -- (5, 2* \y);
	}
\foreach \x in {0,...,4}{
		\draw (-4 + 2* \x, 0) -- (-4 + 2* \x, 6);
		\draw  (-4 + 2* \x, -1) -- (-4 + 2* \x, 0);
		\draw   (-4 + 2* \x, 6)  -- (-4 + 2* \x, 7);
	}
	\end{tikzpicture}\end{minipage}
\begin{minipage}{.35 \textwidth}
\centering
\includegraphics[width=0.75 \textwidth]{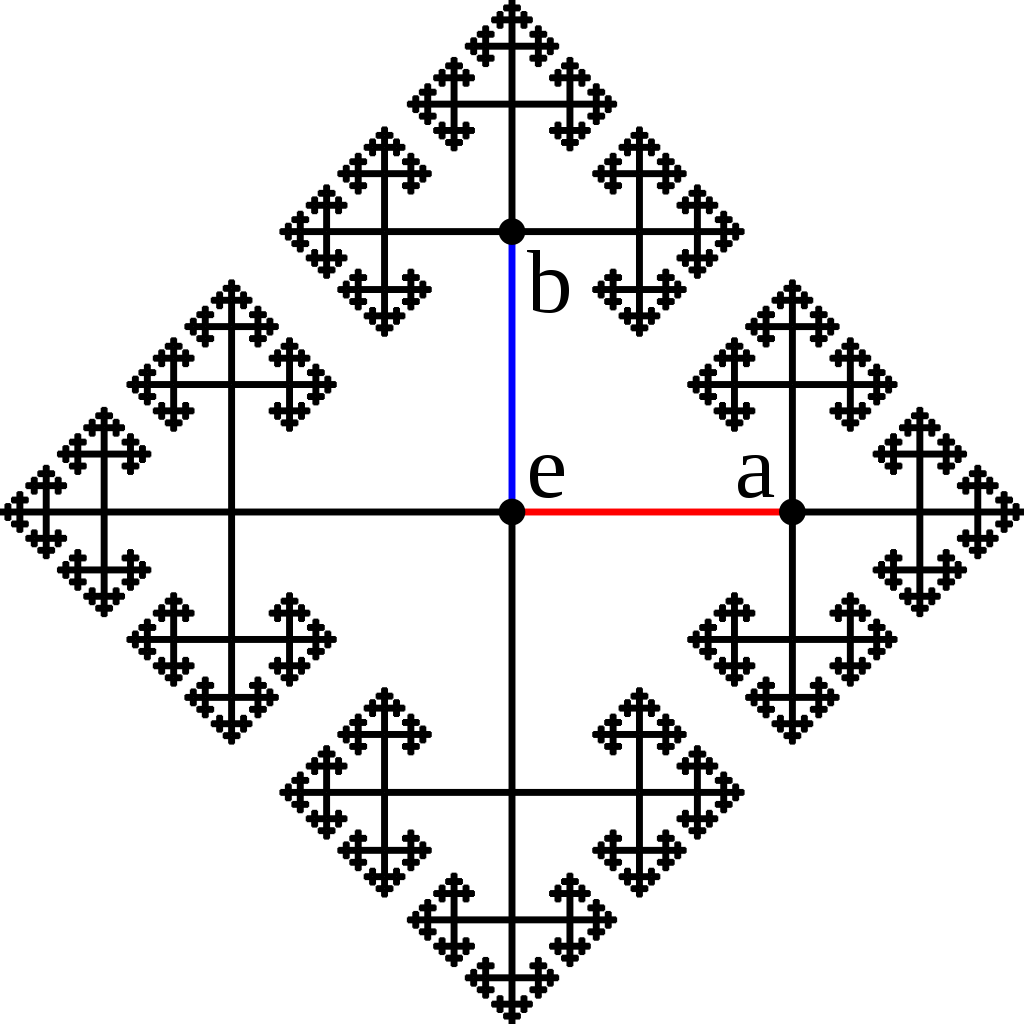}
\end{minipage}
\end{center}
\caption{The three infinite graphs $\cG_d = \Z$, $\cG_d = \Z^2$ and $\cG_d = \mathcal{T}_4$ (the infinite $4$-regular tree). They have two, one and infinitely many ends, respectively.}\label{fig:GraphEnds}
\end{figure}

The approach is based on the notion of {\em ends of metric graphs}. A result by Freudenthal says that every (nice) topological space $X$ can be compactified by adding its so-called topological ends \cite{f31}. For an infinite metric graph $\cG$, these can be defined as follows:  Fix an exhausting sequence $\cG_0 \subset \cG_1 \subset \cG_2 \subset \dots$ of $\cG$ consisting of finite metric subgraphs $\cG_n$, $n \in \N$. A {\em topological end} $\gamma$ of $\cG$ is an infinite decreasing sequence $U_0 \supset U_1 \supset U_2 \dots$ of open subsets $U_n \subseteq \cG$ such that each $U_n$ is a non-compact connected component of $\cG \setminus \cG_n$ (different exhaustions $(\cG_n)_{n \in \N}$ lead to equivalent notions).

The topological ends of $\cG = (\cV, \cE, \ell)$ also admit a combinatorial description in terms of the underlying combinatorial graph $\cG_d = (\cV, \cE)$.  We define a {\em ray} in $\cG_d$ as an infinite sequence of distinct vertices $(v_n)_{n\in\N}$  such that $v_{n}\sim v_{n+1}$ for all $n\in\N$. Two rays $\cR_1,\cR_2$ are called \emph{equivalent}, if there is a third ray meeting both $\cR_1$ and $\cR_2$ infinitely many times. A \emph{graph end of  $\cG_d$} is an equivalence class of rays (see Figure~\ref{fig:GraphEnds}). It turns out that the topological ends of $\cG$ and the graph ends of $\cG_d$ are in bijection: vaguely speaking, for each graph end there is a unique topological end $\gamma = (U_n)_{n \in \N}$ such that the respective rays end up in the sets $U_n$ (see \cite[Section 8.6]{die} or \cite[Section 21]{woe} for details).

Abusing the notation, in the following we will not distinguish between the topological ends of $\cG$ and the graph ends of $\cG_d$, and simply speak of the {\em ends} of $\cG$.

\smallskip

From  the historical point of view, ends of graphs were introduced independently by Freudenthal and Halin. They play an important role in the study of infinite graphs \cite[Chapter~8]{die} and provide one of the simplest boundary notions for compactifications of infinite graphs \cite{woe}. On the other hand, the origins of the notion are closely connected to group theory and the investigations of Freudenthal and Hopf \cite{f31, f44, hop}. Recall that, given a finitely generated group $\mathsf{G}$ and a finite symmetric generating set $S$, the {\em Cayley graph} $\mathcal{C}(\mathsf{G}, S)$ is the graph with vertex set $\cV = \mathsf{G}$ and two elements $x,y \in \mathsf{G}$ are neighbors exactly when $xy^{-1} \in S$. By a result going back to Freudenthal and Hopf, a Cayley graph of an infinite group has either one, two or infinitely many ends, and, moreover, the number of ends is independent of the generating set $S$. A classification of the respective cases in terms of group properties was completed later by Stallings, see e.g.  \cite[Chapter 13]{geog}. For example, the graphs depicted in Figure~\ref{fig:GraphEnds} are Cayley graphs of the groups $\mathsf{G} = \Z$, $\mathsf{G} = \Z^2$ and $\mathsf{G} = \mathbb{F}_2$ (the free nonabelian group of rank two).

\smallskip

In order to take into account the metric (not just topological or combinatorial) structure of the metric graph $\cG$, we introduce the following notion.  An end $\gamma$ of $\cG$ has {\em infinite volume}, if $\operatorname{\vol}(U_n) = \infty$ for all $U_n$ in the corresponding sequence of sets $(U_n)_{n \in \N}$ (this property is independent of the choice of the exhausting sequence $(\cG_n)_{n \in \N}$). Otherwise, $\gamma$ is said to have {\em finite volume}. Here $\operatorname{\vol}(A)$ denotes the Lebesgue measure of a measurable set $A \subseteq \cG$. 

This leads to the following characterization of Markovian uniqueness.

\begin{theorem}[\cite{kmn22}]
Let $\cG = (\cV, \cE, \ell)$ be an infinite metric graph. Then Markovian uniqueness holds if and only if all ends of $\cG$ have infinite volume.
\end{theorem}
We obtain the following corollary in the special case of only one graph end, for instance for graphs arising from (well-behaved) tilings of the plane $\R^2$ or Cayley graphs of amenable groups which are not virtually infinite cyclic.

\begin{corollary}[\cite{kmn22}]
Let $\cG = (\cV, \cE, \ell)$ be an infinite metric graph having only one end. Then Markovian uniqueness holds if and only if $\cG$ has infinite total volume, that is, $\operatorname{vol}(\cG) = \sum_{e \in \cE} \ell(e) = \infty$.
\end{corollary}

In the simple situation when $\cG$ has only finitely many ends of finite volume, one can also use boundary conditions on finite volume ends in order to describe a certain class of self-adjoint restrictions \cite{kmn22,kn21}. Moreover, although Laplacians on metric graphs, Riemannian manifolds and discrete graphs admit many similarities, we are not aware of a geometric characterization of Markovian uniqueness in the other two cases (see \cite[Section 7.2]{kn} for further discussion and references).

\section{Connections to discrete graph Laplacians} \label{sec:DiscretevsContinuous}

In this section, we discuss another interesting feature of quantum graphs: their close relations to \emph{discrete graph Laplacians}. In what follows, let $\cG  = (\cV, \cE, \ell)$ be a metric graph (finite or infinite).

\subsection{Equilateral metric graphs} The connection to discrete Laplacians is most evident for \emph{equilateral metric graphs}, meaning that all edges $e \in \cE$ have unit length $\ell(e) = 1$. In this simple case, the Kirchhoff Laplacian $\bH$ is closely connected to the so-called \emph{normalized Laplacian} on $\cG_d = (\cV, \cE)$ (note also that $\bH$ is self-adjoint, see Section~\ref{ss:SelfadjointnessCriteria}). More precisely, the normalized Laplacian $\rh \colon \ell^2(\cV; \deg) \to \ell^2(\cV; \deg)$ is defined in the Hilbert space 
\[
\ell^2(\cV; \deg) = \Big\{ \f \colon \cV \to \C |\,\, \| \f \|^2_{\ell^2(\cV;\deg)}:= \sum_{v\in\cV} |\f(v)|^2 \deg(v) <\infty \Big\}
\]
with the inner product $( \f, \g ) = \sum_{v \in \cV} \f(v) \bar{\g}(v) \deg(v)$. It maps $\f \in \ell^2(\cV; \deg)$ to $\rh \f \in \ell^2(\cV; \deg)$ given by
\begin{align}\label{eq:LaplNorm}
(\rh \f)(v) = \frac{1}{\deg(v)}\sum_{u\sim v} \big (  \f(v) - \f(u)  \big ), \qquad v \in \cV.
\end{align}
It is easy to check that $\rh$ is a bounded, self-adjoint operator in $ \ell^2(\cV; \deg)$ with spectrum $\sigma(\rh) \subseteq [0,2]$. The normalized Laplacian appears in many areas of mathematics, physics and engineering. It serves as the generator of the simple random walk on $\cG_d$ (i.e., a walker at a vertex $v \in \cV$ jumps to an adjacent vertex $w \sim v$ chosen uniformly at random). In the case of Cayley graphs, the study of connections to algebraic properties of the underlying group was initiated in the work of Kesten \cite{kesten2} (actually, in his PhD thesis).  Moreover, the relationship between the spectral properties of $\rh$ and various graph parameters is one of the core topics within the field of {\em Spectral Graph Theory} (see \cite{chung, cdv} for further details).

In the equilateral case, the Kirchhoff Laplacian and normalized Laplacian are connected by the following simple formula (recall also that $\sigma(\bH) \cup \sigma(\rh) \subseteq [0, \infty)$).

\begin{theorem} \label{thm:Equilateral}
Let $\cG = (\cV, \cE, \ell)$ be an equilateral metric graph. If $ \lambda \ge 0$ and $\lambda \notin \{ (\pi n)^2; n\in \N \}$, then 
	\[
		\lambda \in \sigma (\bH) \quad \Longleftrightarrow \quad  1-\cos  ( \sqrt{\lambda} )  \in \sigma ( \rh)
	\]
for the Kirchhoff Laplacian $\bH$ on the metric graph $\cG = (\cV, \cE, \ell)$ and the normalized Laplacian $ {\rh}$ on the underlying combinatorial graph $\cG_d = (\cV, \cE)$.
\end{theorem}
The equivalence in Theorem~\ref{thm:Equilateral} remains true for specific parts of the spectrum, such as the discrete and essential spectrum, or the absolutely continuous, pure point and singular continuous spectrum. These connections were studied and extended by several authors. We refer to \cite{pan12} and \cite[Section 3.6 and Section 3.8]{bk13} for an overview of the results, history and further references.

If the metric graph $\cG$ is finite, Theorem~\ref{thm:Equilateral} has an elementary proof, which we reproduce below (see also, e.g., \cite[Section 3.6]{bk13}). Note that in this case, the discrete Laplacian $\rh$ is described by a finite, non-negative matrix $\rh \in \R^{\cV  \times \cV}$. In particular, the spectra of both $\bH$ and $\rh$ are discrete and consist only of non-negative eigenvalues. 

\begin{proof}[Proof of Theorem~\ref{thm:Equilateral} for finite graphs]
Consider the eigenvalue equation for the Kirchhoff Laplacian $\bH$,
\begin{equation} \label{eq:ContinuousEigenvalueProblem}
\bH f = \lambda f, \qquad f \in\dom(\bH).
\end{equation}
For an eigenvalue $\lambda\neq 0$, any eigenfunction $f$ must satisfy the equation $-f_e'' = \lambda f_e$ on every edge $e =  [0, \ell(e)]$ of $\cG$. Hence, on an edge $e = [0, \ell(e)]$ with left and right endpoints $v, w \in \cV$,
\[
f_e(x_e) = a_e \cos(\sqrt{\lambda} x_e) + b_e \sin( \sqrt{\lambda} x_e ), \qquad x_e \in e = [0, \ell(e)],
\]
for some $a_e, b_e \in \C$. Clearly, we have $a_e = f (v)$ and $b_e =  f_e'(v) / \sqrt{\lambda}$. Since $\cG$ is equilateral, we obtain the equation
\[
f(w) - \cos(\sqrt{\lambda}) f(v) = f_e'(v) \frac{\sin(\sqrt \lambda)}{\sqrt \lambda}
\]
for all neighbors $w \in \cV$ of some fixed vertex $v \in \cV$. Summing over all neighbors $w$ of $v$, the Kirchhoff conditions imply that
\begin{equation} \label{eq:DiscreteEigenvalueProblem}
\rh \mathbf{f}  = (1 - \cos(\sqrt \lambda)) \mathbf{f}
\end{equation}
for the restriction $\mathbf{f} := f |_{\cV} \colon \cV \to \C$ of $f$ to vertices. In particular, if $f$ does not vanish in all vertices, we arrive at an eigenvalue of the normalized Laplacian $\rh$. However, taking into account the properties of the sinus function, this can only happen if $\lambda = k^2 \pi^2$ for some $k \in \N$. Altogether, we have proven the implication "$\Rightarrow$" in the above equivalence. The converse direction "$\Leftarrow$" follows by performing the above steps in the reverse direction. More precisely, one shows that every function $\mathbf{f} \colon \cV \to \C$ satisfying \eqref{eq:DiscreteEigenvalueProblem} is equal to the restriction $\mathbf{f} = f|_{\cV}$ of a function $f \in \dom(\bH)$ satisfying \eqref{eq:ContinuousEigenvalueProblem}.
\end{proof}

\subsection{The general case} If $\cG$ is not equilateral, then the explicit formula in Theorem~\ref{thm:Equilateral} breaks down. However, it turns out that in terms of \emph{qualitative spectral properties}, the Kirchhoff Laplacian $\bH$ is still connected to a suitable \emph{weighted discrete Laplacian}. Consider the vertex weight $m \colon \cV \to (0, \infty)$ and edge weight $b \colon \cE \to (0, \infty)$ given by
\begin{align} \label{eq:MetricWeights}
&m(v) = \sum_{e \in \cE_v} \ell(e), \quad v \in \cV, &b(e) = \frac{1}{\ell(e)}, \quad e \in \cE.
\end{align}
We define a weighted discrete Laplacian $\rh \colon \dom(\rh) \subseteq \ell^2(\cV;m) \to \ell^2(\cV;m)$ by the difference expression
\begin{equation} \label{eq:DiscreteLaplacian}
(\rh \f) (v)= \frac{1}{m(v)} \sum_{u \sim v}b(e_{u,v}) \big ( \f(v) - \f(u) \big ), \qquad v \in \cV,
\end{equation}
on the domain $\dom(\rh) =\{ \f \in \ell^2(\cV; m) | \, \rh \f \in \ell^2(\cV; m)\}$  in the Hilbert space $\ell^2(\cV; m) = \{ \f \colon \cV \to \C |\,\, \| \f \|^2_{\ell^2(\cV;m)}:= \sum_{v\in\cV} |\f(v)|^2 m(v) <\infty \}$. In the equilateral case, $m \equiv \deg$, $b \equiv 1$, $\dom(\rh) = \ell^2(\cV; \deg)$ and $\rh$ is the normalized Laplacian.

An immediate way of connecting the Kirchhoff Laplacian $\bH$ and the discrete Laplacian $\rh$ arises by noting a connection between harmonic functions. Namely, every harmonic function $f\colon \cG \to \R$  (i.e., $f'' \equiv 0$ on every edge and $f$ satisfies Kirchhoff conditions \eqref{eq:kirchhoff}) must be edgewise affine and satisfy
\[
0 = \sum\limits_{e \in \cE_v}  f_e'(v) = \sum_{ u \sim v} \frac{1}{\ell(e_{u,v})} \big (f(u)- f(v) \big) = (- m \cdot \rh \f) (v)
\]
at each vertex $v\in\cV$. In particular, the restriction to vertices $\mathbf{f} := f |_{\cV} \colon \cV \to \R$ is harmonic, i.e., we have $\rh \f = 0$. Conversely, every harmonic function $\f \colon \cV \to \R$ on $\cV$ gives rise to a harmonic function $f$ on $\cG$ by linear interpolation on edges. This immediately connects, for instance, Liouville-type properties on metric graphs and weighted discrete graphs (and the corresponding Poisson and Martin boundaries). Moreover, the weight $m(v)$ of a vertex $v \in \cV$ equals the Lebesgue mass of the "star" around it (i.e., the union of all incident edges) and this connects the Hilbert spaces $L^2(\cG)$ and $\ell^2(\cV; m)$.

Under the additional assumption that $ \sup_{e \in \cE} \ell(e) <\infty$, the connection between the Laplacians can informally be stated as follows.

\begin{theorem}[\cite{ekmn}]  \label{thm:SpectralEquivalence}  The Kirchhoff Laplacian $\bH$ and the weighted discrete Laplacian $\rh$ have the same "basic spectral properties".
\end{theorem}
For instance, $\bH$ and $\rh$ are self-adjoint simultaneously, spectral gaps are strictly positive simultaneously, and one may connect compactness of the resolvents and properties of the corresponding heat semigroups (see \cite{ekmn} for a detailed list of connections). However, note that Theorem~\ref{thm:Equilateral} provides an explicit formula relating numerical values, whereas the connections in Theorem~\ref{thm:SpectralEquivalence} are only of qualitative nature. In particular, Theorem~\ref{thm:SpectralEquivalence} is only non-trivial for infinite graphs, since qualitative spectral properties are well-known already for Kirchhoff Laplacians and discrete Laplacians on finite graphs. Theorem~\ref{thm:SpectralEquivalence} was proven by Exner, Kostenko, Malamud and Neidhardt, allowing them to further develop spectral theory of infinite quantum graphs by using results on discrete Laplacians \cite{ekmn}. 

\subsection{Spectral theory of graphs:  discrete vs. continuous} A general \emph{discrete graph Laplacian} is of the form \eqref{eq:DiscreteLaplacian} for arbitrary edge and vertex weights $m$ and $b$. Note that weights stemming from metric graphs satisfy
\[
m(v) = \sum_{e \in\cE_v} b(e), \qquad v \in \cV.
\]
From this perspective, Kirchhoff Laplacians on metric graphs correspond to a rather special class of discrete Laplacians. However, one can easily introduce also a weighted version of Kirchhoff Laplacians \cite{kn}. This relates to the approach of using Brownian motion on weighted metric graphs to study random walks on discrete graphs, which has been employed several times in the stochastics literature (see e.g. \cite{baba03, fo14, hush14, var85b} and the references therein).  Namely, in the framework of Dirichlet forms, these two classes of stochastic processes correspond precisely to weighted Kirchhoff Laplacians and discrete Laplacians.

It turns out that Theorem ~\ref{thm:SpectralEquivalence} carries over to the setting of weighted Kirchhoff Laplacians and, allowing in addition graphs with loops, this procedure allows to recover all discrete Laplacians \cite{akmmnn22, kn22a, kn}. Hence naively speaking, in the setting of infinite locally finite graphs and with respect to basic spectral properties,
\begin{center}
"\emph{spectral theory of weighted Kirchhoff Laplacians and discrete Laplacians \\ is equivalent}".
\end{center}
The connections between these two types of operators can be viewed from several different perspectives (such as spectral, parabolic and global metric properties), and exactly their rich interplay leads to important further insight on both sides. For further information on this viewpoint we refer to \cite{akmmnn22, kn22a, kn}.

\ack{I would like to thank Omid Amini, Christof Ender and Aleksey Kostenko for helpful suggestions during the preparation of this article.}

\end{document}